\documentclass[a4paper,10pt,fullpage]{article}
\usepackage{fontenc}
\usepackage{amsmath}
\usepackage{amssymb}
\usepackage{amsfonts}
\usepackage{amsthm}
\usepackage{graphicx}
\usepackage{eucal}
\usepackage{mathrsfs}
\usepackage{amsthm}
\usepackage{verbatim}
\usepackage{epsfig}
\usepackage{hyperref}
\usepackage{float}
\usepackage{color}
\usepackage{mathbbol}

\DeclareFontFamily{OML}{rsfs}{\skewchar\font'177}
\DeclareFontShape{OML}{rsfs}{m}{n}{ <5> <6> rsfs5 <7> <8> <9> rsfs7
  <10> <10.95> <12> <14.4> <17.28> <20.74> <24.88> rsfs10 }{}
\DeclareMathAlphabet{\mathfs}{OML}{rsfs}{m}{n}

\newcommand{\bae}{\begin{equation}\begin{aligned}}
\newcommand{\eae}{\end{aligned}\end{equation}}

\newcommand{\Z}{\mathbb{Z}}
\newcommand{\N}{\mathbb{N}}

\newtheorem{thm}{Theorem}[section]
\newtheorem{prop}[thm]{Proposition}
\newtheorem{lem}[thm]{Lemma}

\newtheorem{definition}{Definition}[section]
\newtheorem{exm}{Example}
\newtheorem{rmk}{Remark}

\begin{document}

\numberwithin{equation}{section}
\numberwithin{figure}{section}
\title{Harmonic Labeling of Graphs}
\author{Itai Benjamini, Van Cyr, Eviatar B. Procaccia, Ran J. Tessler}
\maketitle

\begin{abstract}
Which graphs admit an integer value harmonic function which is  injective and surjective onto $\Z$?
Such a function, which we call harmonic labeling, is constructed when the graph is the $\Z^2$ square grid.
It is shown that for any finite graph $G$ containing at least one edge, there is no harmonic labeling of $ G \times \Z$.
\end{abstract}

\section{Introduction}

Let $G=(V,E)$ be an infinite graph of bounded degree.
\begin{definition}
A function $\phi:V(G)\rightarrow\mathbb{Z}$ is harmonic if
\[
\forall x\in V,\phi(x)=\frac{1}{\deg(x)}\sum_{y\sim x}\phi(y).
\]
If $\phi$ is also injective and surjective we say that $\phi$ is a {\bf harmonic labeling} of $G$.
\end{definition}

In this note we study the somewhat ad-hoc question, which graphs admit harmonic labeling?
In view of some of the techniques used, we hope that a nice structure will emerge,
making this project natural and interesting. The study  linking between the behavior of harmonic functions on spaces
and geometric properties of the spaces, started with Liouville theorem, is well developed, see e.g. \cite{yau1994lectures}. Here we want to suggest a variant in the spirit of additive combinatorics over graphs, see e.g. \cite{alon2009sums}.
\medskip

In the coming sections we will prove some positive and negative results regarding the existence of harmonic labeling. We are still far from classifying Cayley graphs or vertex transitive graphs, with respect to this property or understanding  the structure of such functions. In particular we don't know if admitting harmonic labeling is a quasi isometric invariance for Cayley graphs? It will be of interest to verify on natural examples of Cayley graphs whether they admit harmonic labeling and come up with useful invariants. In the last section we list few open problems.
Before moving to trees, grids and products, here are two quick warm up examples.

\begin{exm}
The easiest example of harmonic labeling is when $V(G)=\mathbb{Z}$ and $E(G)=\{(i,j):|i-j|=1\}$. Let $\phi\equiv {\rm Id}$. Each vertex is of degree 2, and $i=\phi(i)=\frac{1}{2}(\phi(i+1)+\phi(i-1))=\frac{1}{2}(i+1+i-1)=i$. For a graphical representation of the $\mathbb{Z}$-labeling see Figure \ref{fig:zlabeling}. In the graphical representation we write the image of $\phi$ on the vertices of the graphs.
\begin{figure}[H]
\begin{center}
\includegraphics[width=0.7\textwidth]{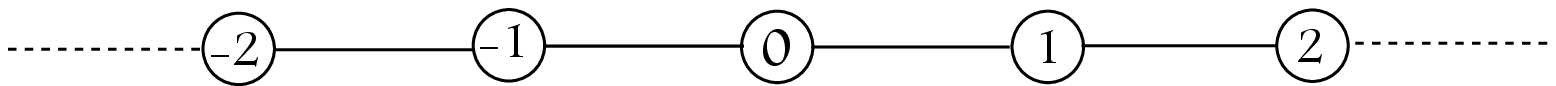}
\caption{Labeling on $\mathbb{Z}$\label{fig:zlabeling}}
\end{center}
\end{figure}
\end{exm}
The next example shows that the existence of harmonic labeling is not trivial.
\begin{exm}
Let $G$ be the infinite cross, as in Figure \ref{fig:crosslabeling}. If $\phi$ is a harmonic labeling then so are $\psi(x)=\phi(x)-\phi(0)$ and $\psi(x)=-\phi(x)$. Thus we can always assume zero value on some appointed vertex. Assume by contradiction that a harmonic labeling of the cross exists, and assume zero value at the intersection. From the harmonic property three values $a,~b,~c$ span the value of $\phi$ on the cross. We can assume that two of these are positive (if not look at $\psi(x)=-\phi(x)$). Without loss of generality assume $a$ and $b$ are positive. We will obtain the number $a\cdot b$ in both the left and the right arms of the cross contradicting the assumption $\phi$ is injective.
\begin{figure}
\begin{center}
\includegraphics[width=0.7\textwidth]{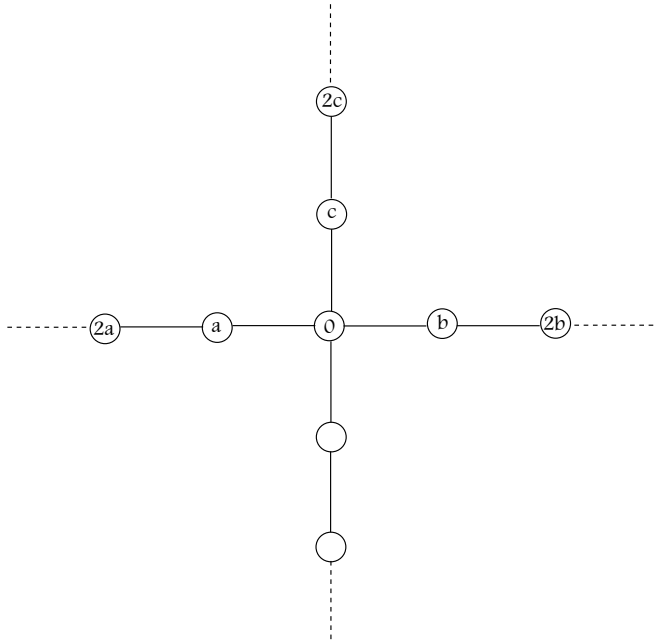}
\caption{No labeling on the cross\label{fig:crosslabeling}}
\end{center}
\end{figure}
\end{exm}

From the examples and the the study below, the notion of label spanning sets (Remark \ref{rmk:span}, Section \ref{sec:openprob}) emerges as a key concept. It is of interest to study and classify these sets.

\section{Trees}
\begin{thm}
Let $G$ be a $d$-regular tree, then there exists a labeling on $G$.
\end{thm}
\begin{proof}
We start with the four-regular tree. Observe Figure \ref{fig:4regulartree}.
\begin{figure}
\begin{center}
\includegraphics[width=0.5\textwidth]{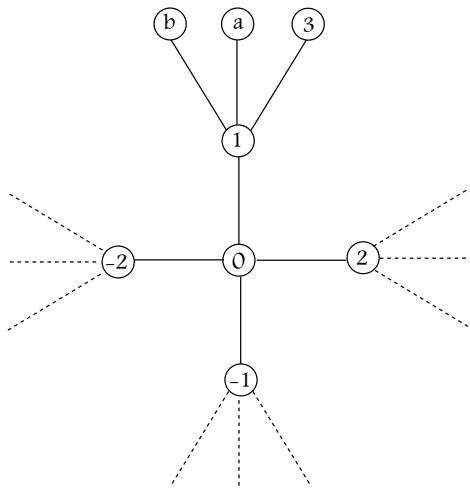}
\caption{Four-regular tree\label{fig:4regulartree}}
\end{center}
\end{figure}

We build $\phi:V(G)\rightarrow\mathbb{Z}$ recursively. Assume $\phi$ is defined on some subset $V'\subset V(G)$. Let \[m=\min_{z\in\mathbb{Z}}\{|z|:z\not\in\phi(V')\}.\], if $m\in\phi(V')$ take $-m$. Define $\phi$ to be $m$ on one of the free leaves on the boundary of $V'$. The parent of this vertex has two other children, denote by $f=\phi({\rm parent})$, $f=\phi({\rm parent~of~parent})$. Those children give us the required degree of freedom to make $\phi$ injective. Define $\phi$ on one of them to be a number $a$ much larger than \[\max_{v\in V'}|\phi(v)|.\] Thus all that is left is to show that the value of the second child $b$ is different than all the previous values of $\phi$.
\[
b=4*f-m-a-f'.
\]
It is easy to see that if $a$ is large enough $b$ is small enough and $\phi$ is injective. Surjectivity is obvious from construction.

Notice that for $d>4$ we will have more degrees of freedom, and the recursive construction will work as well. For the three-regular tree the previous construction fails since there are not enough degrees of freedom in each level of the tree. Thus the construction will use two levels in each step. Observe Figure \ref{fig:3regulartree}. We build $\phi$ recursively, very similar to the $d>3$ case. Assume $\phi$ is defined for some subset $V'\subset V(G)$. We wish to define $\phi$ on one of the leaves with minimal distance to the root. Let $m$ be as in the previous construction, a number with minimal absolute value that is not in $\phi(V')$. Define $\phi$ to be $b$ in the desired leaf, $m$ on one of its children and $b$ on the second child. Let $f=\phi({\rm parent})$, then $b=\frac{1}{3}(a+f+m)$. If $a$ is much larger than all the values in $\phi(V')$, $b$ will be much smaller and $\phi$ stays injective. Once again $\phi$ is surjective by construction.
\begin{figure}
\begin{center}
\includegraphics[width=0.7\textwidth]{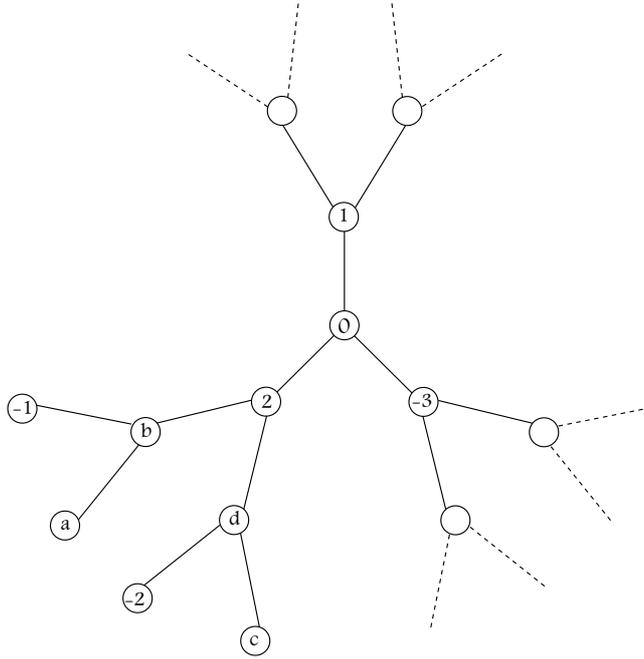}
\caption{Three-regular tree\label{fig:3regulartree}}
\end{center}
\end{figure}
\end{proof}

\begin{thm}\label{thm:everytree}
Every tree such that all vertices have degree greater than 2 has a harmonic labeling.
\end{thm}
\begin{proof}
Choose a vertex and denote it as the origin $\Gamma(0)$. Let \[\Gamma(n)=\{\Gamma(n)_i^j:i=1,\ldots,|\Gamma(n-1)|,j=1,\ldots,{\rm degree~ of~ common~ father}-1\}\] be the set of vertices of distance $n$ from the origin. We build the harmonic labeling recursively. Suppose we defined $\phi$ on $\Gamma(n-1)$ lets expand it to $\Gamma(n)$. Let $l$ be the degree of common father of $\Gamma(n)_1^j$, and let $m_1,\ldots,m_{l-2}$ be the numbers in $\N$ with smallest absolute value such that $m_1,\ldots,m_{l-2}\notin\cup_{i=0}^{n-1}\Gamma(i)$. Define $\Gamma(n)_1^j=m_j$ for $j=1,\ldots,l-2$, now define $\Gamma(n)_1^{l-1}$ be big enough such that both it and $\Gamma(n)_1^l$ be different than $\cup_{i=0}^{n}\Gamma(i)\setminus \Gamma(n)_1^{l-1}\cup\Gamma(n)_1^{l}$. We continue the same way for $\Gamma_i^j$, $i\neq1$.

\end{proof}

\section{Finite graphs }
\begin{thm}\label{thm:finite}
Let $G$ be a finite graph with at least one edge. The graph $G\times\mathbb{Z}$ does not admit harmonic labeling.
\end{thm}
As an example to this theorem we will show that there is no harmonic labeling for the ladder graph.
\begin{prop}
The graph $\mathbb{Z}\times\mathbb{Z}_2$ does not admit harmonic labeling.
\end{prop}
\begin{proof}
Assume by contradiction that there exists a harmonic labeling \[\phi:\mathbb{Z}\times\mathbb{Z}_2\rightarrow\mathbb{Z}.\]
In Figure \ref{fig:ladder} we denote by
\bae
\{a_i\}_{i\in\Z}&=\{\phi(i\times\{0\})\}_{i\in\Z}\\
\{b_i\}_{i\in\Z}&=\{\phi(i\times\{1\})\}_{i\in\Z}.
\eae
\begin{figure}
\begin{center}
\includegraphics[width=0.7\textwidth]{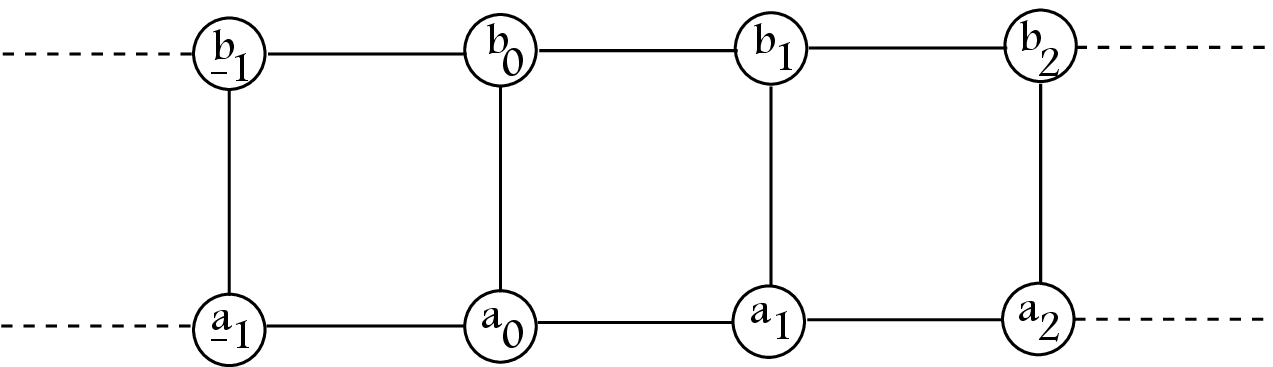}
\caption{Ladder graph $\Z\times\Z_2$\label{fig:ladder}}
\end{center}
\end{figure}
By the harmonic property of the labeling, $\phi$ is determined by its values $a_0,a_1,b_0,b_1$. Without loss of generality assume $a_1=\max\{a_0,a_1,b_0,b_1\}$. Since we assumed that $\phi$ is injective, $b_1$ and $a_0$ are smaller than $a_1$ by at least $1$ and $2$. This yields
\bae
a_2=3a_1-b_1-a_0\geq 3a_1-(a_1-1)-(a_1-2)=a_1+3.
\eae
This means that either $a_2$ or $b_2$ are the maximum of $\{a_i,b_i\}_{i=0}^{2}$ and greater than $a_1+3$. By induction we get that for every $k\in\mathbb{N}\setminus\{0\}$ either $a_k$ or $b_k$ are the maximum of $\{a_i,b_i\}_{i=0}^{k}$ and greater than $\max\{a_i,b_i\}_{i=0}^{k-1}+3$. Assume that at least one of the next limits are true $\liminf_{n\rightarrow\infty}a_n=-\infty$, $\liminf_{n\rightarrow\infty}b_n=-\infty$. Without loss of generality assume the second limit is true. There is some $N_1\in\mathbb{N}$ such that $b_{N_1}=\min\{a_i,b_i\}_{i=0}^{N_1}$. Now for all $k\in\mathbb{N}$, $b_{N_1+k}\leq b_{N_1+(k-1)}-3$ and $a_{N_1+k}\geq a_{N_1+(k-1)}+3$. If non of the above limits are true the the right side of the ladder is bounded from below and this leaves us with two options.
\begin{enumerate}
\item The maximum value changes infinitely many times between $a$ and $b$.
\item From some $i$, $b_i$ is the minimum of $a_i,b_i,a_{i+1},b_{i+1}$, and thus for each $j>i$ $b_{j}\leq b_{j+1}-3$.
\end{enumerate}
In any case The density of positive integers in the right part of the ladder is less than $2/3$.

If without loss of generality $\limsup_{i\rightarrow-\infty}a_i=\infty$, there exists some $N_2\in-\mathbb{N}$ such that $a_{N_2}=\max\{a_i,b_i\}_{i=-1}^{N_1}$ and $b_{N_2}=\min\{a_i,b_i\}_{i=-1}^{N_1}$. As before for all $k\in\mathbb{N}$, $b_{N_1-k}\leq b_{N_1-(k-1)}-3$ and $a_{N_1-k}\geq a_{N_1-(k-1)}+3$. Thus the density of negative integers is no more than $2/3$. This contradicts our assumption that $\phi$ is surjective.
\end{proof}

\begin{proof}[Proof of Theorem \ref{thm:finite}]
By the harmonic property, determining $\phi$ on two adjacent slices of $G$ determines the function on all $G\times\Z$. For a graph $G$ with $n$ vertices define the Laplacian matrix $L=(l_{i,j})_{n\times n}$
\bae
l_{i,j}=\left\{\begin{array}{ll}
{\rm deg}(v_i)&\quad{\rm if}~i=j\\
-1&\quad {\rm if}~i\sim j\\
0&\quad {\rm otherwise}
\end{array}\right.
.\eae
$L$ is a real symmetric matrix, and hence has a real orthogonal diagonalization.
Let $\mu_0\leq\mu_1\leq\ldots\leq\mu_{n-1}$ be the eigenvalues of $L$. It is known that $\forall i,\mu_i\geq0$, $\mu_0=0$ (the vector $(1,1,\ldots,1)$ is an eigenvector with eigenvalue $0$). The multiplicity of the eigenvalue $0$ is the number of connected components in $G$, and then for every connected component there is the corresponding $0$-eigenvector which is the vector all of whose entries are $0$, accept for those which stand for vertices in that connected component - and their value is $1$.

Denote by $\overline{y}=\{y_i\}_{i=1}^n$ the values of $\phi$ on some slice of $G$ and $\overline{x}=\{x_i\}_{i=1}^n$ the values on the slice above it.
\begin{lem}
Let $\overline{a},\overline{b},\overline{c}$ be the values of $\phi$ on three slices of $G$ one on top of the other.
Then
\bae
\left(\begin{array}{ll}
\overline{c}\\
\overline{b}
\end{array}\right)=\left(\begin{array}{ll}
L+2I&\quad -I\\
I&\quad 0
\end{array}\right)\left(\begin{array}{ll}
\overline{b}\\
\overline{a}
\end{array}\right)
\eae
\end{lem}
\begin{proof}
Immediate from the harmonic property, note that we add $2I$ because the degree of each vertex in $G$ increases by $2$ after multiplying by $\Z$. See Figure \ref{fig:FinitegraphtimesZ} for clarification.
\begin{figure}
\begin{center}
\includegraphics[width=0.45\textwidth]{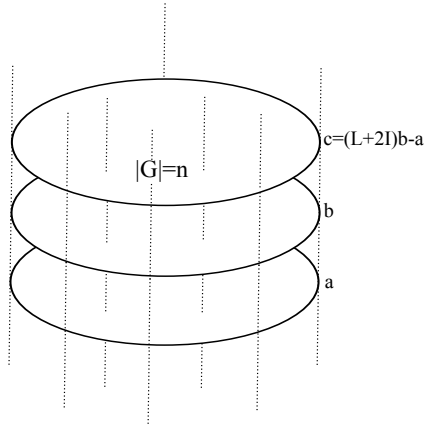}
\caption{Finite graph $G$ times $\Z$\label{fig:FinitegraphtimesZ}}
\end{center}
\end{figure}

\end{proof}
Next we analyze the eigenvalues of $A= \left(\begin{array}{ll}
L+2I&\quad -I\\
I&\quad 0
\end{array}\right)$. Assume $\lambda$ is an eigenvalue of $A$ with an eigenvector $\left(\begin{array}{ll}
\overline{x}\\
\overline{y}
\end{array}\right)$. Then

\bae\label{eq:eigencalctwotoone}
\left(A-\lambda I\right)\left(\begin{array}{ll}
\overline{x}\\
\overline{y}
\end{array}\right)=0&\Leftrightarrow
\left(\begin{array}{ll}
L+(2-\lambda)I&\quad -I\\
I&\quad -\lambda I
\end{array}\right)\left(\begin{array}{ll}
\overline{x}\\
\overline{y}
\end{array}\right)=0\\
&\Leftrightarrow
\left\{\begin{array}{ll}
\overline{x}=\lambda \overline{y}\\
L\overline{x}+(2-\lambda)\overline{x}=\overline{y}
\end{array}\right.\\
&\Rightarrow L\overline{x}=\left(\lambda+\frac{1}{\lambda}-2\right)\overline{x}.
\eae
Thus $\overline{x}$ is an eigenvector of $L$ with eigenvalue $\mu=\lambda+\frac{1}{\lambda}-2$. On the other direction, if $\mu$ is an eigenvalue of $L$ with eigenvector $\overline{a}$ then $\left(\begin{array}{ll}
\overline{a}\\
\lambda\overline{a}
\end{array}\right)$ and $\left(\begin{array}{ll}
\overline{a}\\
\frac{1}{\lambda}\overline{a}
\end{array}\right)$ are eigenvectors of $A$ with eigenvalues $\lambda$ and $\frac{1}{\lambda}$, that satisfy $\mu=\lambda+\frac{1}{\lambda}-2$. We know $\mu\geq0$ so $\lambda,\frac{1}{\lambda}=\frac{1}{2}\left((\mu+2)\pm\sqrt{(\mu+2)^2-4}\right)>0.$
\\ In particular, for $\mu\neq 1$, there are two different eigenvalues $\lambda, \frac{1}{\lambda} \neq 1$.
\\ For $\mu = 0$, $\lambda = 1$, the corresponding eigenvector is $\left(\begin{array}{ll}
\overline{v}\\
\overline{v}
\end{array}\right)$.
where $\overline{v}$ is the eigenvector of $L$ corresponding to the eigenvalue $0$ and the relevant connected component.
\\ It is easy to verify that for any such $\overline{v}$, we have that
\bae
A\left(\begin{array}{ll}
\overline{v}\\
\overline{0}
\end{array}\right)= \left(\begin{array}{ll}
2\overline{v}\\
\overline{v}
\end{array}\right) = \left(\begin{array}{ll}
\overline{v}\\
\overline{0}
\end{array}\right) + \left(\begin{array}{ll}
\overline{v}\\
\overline{v}
\end{array}\right)
\eae

Thus, we have found a Jordan decomposition for $A$. For any eigenvalue of $A$, different from $1$, we know its eigenvector, and for $\lambda = 1$ we have found the eigenvector and the extended eigenvector (there might be several, due to multiplicity). The above vectors, together, form a basis for the vector space, which is of dimension $2n$, where $n$ is the number of vertices of $G$.

Let $\{v_i\}$ be the eigenvectors and extended eigenvectors of $A$ with eigenvalues $\lambda_i$. Any vector of $2n$ coordinates can be represented as a linear combination of these vectors, and in particular any vector which represents the values of $\phi$ on two consecutive slices.

\begin{lem}\label{lem:nooneeigen}
The representation of vectors which represent the values of $\phi$ on two consecutive slices, using the base $\{v_i\}$, uses only eigenvectors or extended eigenvectors with eigenvalue $\lambda= 1$.
\end{lem}

Assuming Lemma \ref{lem:nooneeigen} we can finish the proof of the theorem. Indeed, the basis elements which may be used for the representation of such vectors, must be constant on every set of coordinates which corresponds to a connected component of $G$ (note that for any such component, there are two sets of coordinates, among the $2n$ coordinates, that correspond to that component - as there are two slices). This is due to the fact that both the eigenvectors of $1$ (in $A$) and the extended (Jordan) eigenvectors are constant on such sets. Thus, any combination of such vectors will be constant on connected components (in any slice separately), but since there is at least one edge in $G$, there is at least one such component with at least two vertices. $\phi$ must receive the same value for both vertices - in contradiction with $\phi$ being bijective.

We now continue to the proof of the lemma.
\begin{proof}[Proof of Lemma \ref{lem:nooneeigen}]
It is enough to prove there are no eigenvectors for eigenvalues larger then $1$ in the representation of such a vector $\left(\begin{array}{ll}
\overline{x}\\
\overline{y}
\end{array}\right)$
as inverting the order of the slices takes an eigenvalue $\lambda$ to $\frac{1}{\lambda}$ and vice versa and leaves us with another harmonic labeling. To see this remember equation \eqref{eq:eigencalctwotoone} and calculate
\bae
A\left(\begin{array}{ll}
\overline{y}\\
\overline{x}
\end{array}\right)&=
\left(\begin{array}{ll}
L+(2-\lambda)I&\quad -I\\
I&\quad -\lambda I
\end{array}\right)\left(\begin{array}{ll}
\overline{y}\\
\overline{x}
\end{array}\right)\\
&=\left(\begin{array}{ll}
L\overline{y}+2\overline{y}-\overline{x}\\
\overline{y}
\end{array}\right)\\
&=\left(\begin{array}{ll}
L\frac{1}{\lambda}\overline{x}+2\frac{1}{\lambda}\overline{x}-\frac{\overline{y}}{\lambda}+\frac{\overline{y}}{\lambda}-\lambda\overline{y}\\
\frac{1}{\lambda}\overline{x}
\end{array}\right)=
\frac{1}{\lambda}\left(\begin{array}{ll}
\overline{y}\\
\overline{x}
\end{array}\right)
\eae
We are left with showing that if $\overline{u}=\left(\begin{array}{ll}
\overline{x}\\
\overline{y}
\end{array}\right)$ represents the values of $\phi$ in two consecutive slices, then in it's linear combination with respect to the basis $\overline{v_i}$, no eigenvector for eigenvalue greater than $1$ appears.
We define a $\emph{line}$ to be a set of vertices in $G\times\mathbb{Z}$ with the same $G$-coordinate. We have $n$ lines.
\begin{definition}
Let $S = \{s_n\}_{n=-\infty}^{n=\infty}$ be a monotone increasing sequence of (different) integers.
Denote by $[n]$ the set $\{-n,...,n-1,n\}$. If \[\lim_{m\rightarrow\infty}\left|[m]\bigcap S\right|/(2m+1)\] exists, it is called the $\emph{density}$ of $S$.
\end{definition}
Some well known (and easy to establish) properties of the density are the following:
\begin{enumerate}
\item An arithmetic sequence with more than one element has positive density.
\item A sequence which grows in absolute value exponentially (in both directions) has density $0$.
\item $\mathbb{Z}$ has density $1$.
\item The density of the union of two sequences (reordered to be monotonic) is no more than the sum of densities of the two sequences.
\end{enumerate}

Assume that
\bae
\overline{u} = \sum_{i=1}^{2n}a_i \overline{v_i}
\eae
Then in particular, for the $j^{th}$ coordinate $u^j = \sum_{i=1}^{2n}a_i {v_i}^j$.
We know that the values of $\phi$ in the slices "higher" then those represented by $\overline{u}$ can be calculate by repeated iterations of the matrix $A$. Thus, if the last ("lower") $n$ coordinates of $\overline{u}$ represent slice number $0$, and the first $n$ coordinates represent slice number $1$, then $\phi$ takes in slice number $k$ the values which appear in the first $n$ coordinates of $\overline{u_k} := A^k \overline{u}$ (or equivalently - those in the last $n$ coordinates of
$A^{k+1} \overline{u}$). Thus, for any eigenvector $\overline{v_i}$ for eigenvalue $\lambda_i \neq 1$, its coefficient in the representation of $\overline{u_k}$ is ${\lambda_i}^{k} a_i$.

Let $\lambda_i > 1$ be a maximal eigenvalue for which $a_i \neq 0$. There must be at least one such $\lambda_i$, due to the assumption of the lemma, yet - there might be several equal eigenvalues. Denote by $t^{j}$ the sum $\sum_{\{l | \lambda_l = \lambda_i\}}a_l {v_l}^{j}$, where $v_l^j$ denotes the $j$-th coordinate of $v_l$. This may be considered as the sum of contributions of the eigenvectors with eigenvalue $\lambda_i$ to $\overline{u}$ in the $j^{th}$ coordinate. As the eigenvectors are linearly independent, and $a_i \neq 0$, not all the $t^{j}$'s can be $0$. Without loss of generality $t^{1} \neq 0$.

Note that there must be some other eigenvalue $\lambda_b < 1$ such that the sum of contributions of eigenvectors with that eigenvalue $\lambda_i$ to $\overline{u}$ in the $1^{st}$ coordinate $r^1:= \sum_{\{l | \lambda_l = \lambda_b\}}a_l {v_l}^{1}$ is non-zero as well. The reason is that if we consider lower slices (applying $A^{-k}$) then the contributions of eigenvectors with eigenvalues smaller than $1$ decay exponentially and the contributions of eigenvectors and extended eigenvectors of $1$ to the first coordinate are summed to be an arithmetic progression (that need not to be, at least naively, an arithmetic progression of integers). If there are no other elements added, then the $1^{st}$ coordinate in "negative" slices are only the sum of such sequences. But it is very easy to verify that such a sum cannot be always integer (as it must be) if the exponential decaying sequences do not add to $0$. And in our case they cannot.

Thus, for a large enough positive integer $k~$ ${u_k}^{1}\sim t^{1} {\lambda_i}^k$, and for small enough negative integer $k~$ ${u_k}^{1}\sim r^{1} {\lambda_b}^k$ (assuming that $\lambda_b$ is the smallest eigenvalue with nonnegative contribution of its eigenvectors to the $1^{st}$ coordinate). Hence, from property 2 of densities, we have that the values that the first coordinate receives along its line is a set of density $0$ in $\mathbb{Z}$.
This analysis can be carried to any coordinate with nonzero contribution related to some eigenvalue other than $1$.

Assume that there are $k$ coordinates without any contribution from eigenvalue
$\lambda \neq 1$, and the rest $n-k$ coordinates are contributed from such eigenvalues. From property 4 of the density it follows that the density of the union of these $n-k$ lines is $0$. On the other hand, in the other $k$ lines, the coordinates (values of $\phi$) are non constant arithmetic progressions (due to the Jordan extended eigenvectors). Thus - they have a density (property 1).
The sets of values of $\phi$ in each line are pairwise disjoint ($\phi$ is bijective) and their union cover $\mathbb{Z}$ ($\phi$ is surjective). Thus, due to properties 4,3 the density of the union of the $k$ arithmetic progressions is $1$. We shall now prove the following lemma, and use it to finish the argument.
\begin{lem}
If $S^1 ,...S^n$ are arithmetic sequences s.t. the density of their union is $1$ then they cover all $\mathbb{Z}$.
\end{lem}
\begin{proof}
Assume that $S^i$ has the form $\{a_i+md_i\}_m$. Let $D$ denote the l.c.m. of $\{d_i\}$.
Consider the intersection of $\{1,2,...,D\}$ with $S$, the union of the sequences.
If this intersection is all of $\{1,2,...,D\}$, then $S$ contains all of $\mathbb{Z}$. Indeed, if $k\in S$ then $k+mD$ is in $S$ as well, for any integer $m$, as $k$ must belong to some $S^i$, hence $k+ld_i$ must be in $S$ for any $l$, but since $D$ is a multiple of $d_i$, $k+mD$ is in $S$ as well, for any $m$. Thus $S\supseteq \mathbb{Z}$, and hence $S = \mathbb{Z}$.

On the other hand, if there is some $1 \leq k \leq D$ not in $S$, a similar argument shows that $k+mD$ is never in $S$, for any $m$ (as the analysis above shows that if $k+mD$ is in $S$, so is $(k+mD)-mD$). But from here it is easy to verify that the density of $S$ is at most $1-\frac{1}{D}$. A contradiction.
\end{proof}
We can now finish the proof of Lemma \ref{lem:nooneeigen}. It follows from the lemma that the values of $\phi$ in the relevant $k$ lines (those without exponential growth), which are arithmetic progressions with total density (density of the union) being $1$, must cover $\mathbb{Z}$ by themselves. But then since the rest of the $n-k$ lines also have values in $\mathbb{Z}$, some values must repeat, in contradiction with $\phi$ being bijective.
And the lemma is proven.
\end{proof}
And the theorem is thus proved.
\end{proof}

\begin{rmk}
It is easy to see that if $G$ is just a set of points, without any edges, then there exists a harmonic labeling - $G\times\mathbb{Z}$ is a collection of $d$ copies of $\mathbb{Z}$, name them $0,1,2,...,d-1$ and then $\phi (a,z) := a+dz$, where $a \in \{0,1,2,...,d-1\}$, $z \in \mathbb{Z}$ represent the $z^{th}$ point in the $a^{th}$ copy of $\mathbb{Z}$.
\end{rmk}

\begin{rmk}Note that $\Z^2\times\Z_2$ has a harmonic labeling and a regular tree times $\Z_2$ also has a harmonic labeling (easy generalization of theorems \ref{thm:everytree} and \ref{thm:ztwolabeling}) . If one finds a finite graph $G$ such that $\Z^2\times G$ doesn't have a harmonic labeling it also provides an example to an infinite graph $G'=\Z\times G$ such that $G'\times\Z$ doesn't have a harmonic labeling.
\end{rmk}

\begin{rmk}\label{rmk:injecladder}
From the proof of Theorem \ref{thm:finite}, one can derive a construction for an injective harmonic function from the ladder graph to $\Z$. Calculating the Eigenvalues and Eigenvectors of the ladder's matrix $A$, we find there is an Eigenvector with an Eigenvector greater than $1$ and another with an Eigenvalue smaller than $1$. These Eigenvectors grow exponentially in one direction of the ladder and decreases exponentially in the other and in opposite directions. There exists a linear combination of the two Eigenvectors that assures injectivity.  
\end{rmk}

\section{Labeling of $\mathbb{Z}^d$}
\begin{thm}\label{thm:ztwolabeling}
There is a harmonic labeling of $\mathbb{Z}^2$.
\end{thm}
Before presenting the rigorous proof we will present a sketch of the proof. The labeling $\phi$ is spanned by its values on $\Z\times\{0\}\cup\Z\times\{-1\}$. The construction is simply placing a rapidly increasing sequence on one copy of $\Z$ and a filling sequence on the other i.e. a sequence that takes care of the surjectivity of $\phi$, grows slowly and jumps only over previously introduced values of $\phi$. It is left to show that $\phi$ is injective. We get this result with tedious calculations, but the idea is simple. For every large element of the sequence there are only a linear number of other vertices with the same order of magnitude and they all have different coefficients in front of the smaller magnitudes.
\begin{rmk}\label{rmk:span}
Notice that by Theorem \ref{thm:ztwolabeling} $\Z\times\Z$ admits harmonic labeling where as $\Z\times[-k,k]$ does not admit harmonic labeling for all $k$ (by Theorem \ref{thm:finite}). The essential difference between the two graphs is that a harmonic function on $\Z\times[-k,k]$ is spanned by a finite set and a harmonic function on $\Z^2$ is spanned by an infinite set, leaving the function with more degrees of freedom that allow it to be bijective. See Section \ref{sec:openprob} (Question \ref{item:spanningquestion}) for an open question inspired by this remark.
\end{rmk}
\begin{proof}[Proof of Theorem \ref{thm:ztwolabeling}]
We will construct a harmonic labeling of $\mathbb{Z}^2$ inductively.  To do this we need some notation.

Let $A(x):=(x,-1)\in\mathbb{Z}^2$ and $B(x):=(x,0)\in\mathbb{Z}^2$.  Define
\begin{eqnarray*}
A&:=&\bigcup_{x\in\mathbb{Z}}A(x) \\
B&:=&\bigcup_{x\in\mathbb{Z}}B(x).
\end{eqnarray*}
For $n\in\mathbb{N}$ and $0\leq i\leq n$, set \\ \\
\centerline{\begin{tabular}{ll}
$UL_n^i:=(-n+i,i)$&``Upper-left'' \\
$UR_n^i:=(n+1-i,i)$ & ``Upper-right''  \\
$LL_n^i:=(-n+i,-i-1)$ & ``Lower-left''  \\
$LR_n^i:=(n-i+1,-i-1)$ & ``Lower-right.''
\end{tabular}} \\
Set
\begin{eqnarray*}
UL(n)&:=&\bigcup_{0\leq i\leq n}UL_n^i \\
UR(n)&:=&\bigcup_{0\leq i\leq n}UR_n^i \\
LL(n)&:=&\bigcup_{0\leq i\leq n}LL_n^i \\
LR(n)&:=&\bigcup_{0\leq i\leq n}LR_n^i \\
\end{eqnarray*}
and finally let
$$
S_n:=\bigg(\bigcup_{0\leq k\leq n}UL(k)\bigg)\cup\bigg(\bigcup_{0\leq k\leq n}UR(k)\bigg)\cup\bigg(\bigcup_{0\leq k\leq n}LL(k)\bigg)\cup\bigg(\bigcup_{0\leq k\leq n}LR(k)\bigg).\footnote{The above terminology for $UL(n), UR(n), LL(n)$ and $LR(n)$ as ``upper-left,'' ``upper-right,'' etc. refers to its location on the boundary of the set $S_n$.}
$$
Note that $S_{n+1}=S_n\cup UL(n+1)\cup UR(n+1)\cup LL(n+1)\cup LR(n+1)$.

Observe that, for any harmonic function $f:=\mathbb{Z}^2\to\mathbb{R}$, the value of $f$ at the point $(x,y)$ is an integer combination of the numbers $f(x,y-1), f(x-1,y-1), f(x+1,y-1)$ and $f(x,y-2)$.  Inductively, an arbitrary function $\widetilde{f}:A\cup B\to\mathbb{Z}$ extends uniquely to a harmonic function $f:\mathbb{Z}^2\to\mathbb{Z}$.  Moreover, the values of $\widetilde{f}$ on the set $\{A(x):-n\leq x\leq n+1\}\cup\{B(x):-n\leq x\leq n+1\}$ determine the values of $f$ on the set $S_n$ (but does not determine the value of $f$ at any point in $\mathbb{Z}^2\setminus S_n$).

We will construct a sequence of functions
$$
\widetilde{f}_n:\{A(x):-n\leq x\leq n+1\}\cup\{B(x):-n\leq x\leq n+1\}\to\mathbb{Z}
$$
so that for each $n$, $\widetilde{f}_{n+1}$ extends the function $\widetilde{f}_n$ and will denote by $f_n:S_n\to\mathbb{Z}$ the (unique) function determined by $\widetilde{f}_n$ using the harmonic property.  We will do this in such a way that the common extension $f:\mathbb{Z}^2\to\mathbb{Z}$ of all of the functions $f_n$ is harmonic and bijective. \\

\noindent {\em Step 1 (Base of Induction):}
Begin by setting
\begin{eqnarray*}
\widetilde{f}_1(UL_0^0)&:=&\phantom{-}10! \\
\widetilde{f}_1(UL_1^0)&:=&\phantom{-}10!! \\
\widetilde{f}_1(UR_0^0)&:=&\phantom{-}10!!! \\
\widetilde{f}_1(UR_1^0)&:=&\phantom{-}10!!!! \\
\widetilde{f}_1(LL_0^0)&:=&\phantom{-}0 \\
\widetilde{f}_1(LL_1^0)&:=&-1 \\
\widetilde{f}_1(LR_0^0)&:=&\phantom{-}1 \\
\widetilde{f}_1(LR_1^0)&:=&\phantom{-}2.
\end{eqnarray*}

By the harmonic property,
\begin{eqnarray*}
f_1(UL_1^1)&=&\phantom{-}4\cdot10!-10!!-10!!! \\
f_1(UR_1^1)&=&\phantom{-}4\cdot10!!!-10!-10!!!!-1 \\
f_1(LL_1^1)&=&-10! \\
f_1(LR_1^1)&=&\phantom{-}4-10!!!-2.
\end{eqnarray*}

\noindent {\em Step 2 (Induction Step):}
Suppose we have constructed the function
$$
\widetilde{f}_n:\{A(x):-n\leq x\leq n+1\}\cup\{B(x):-n\leq x\leq n+1\}\to\mathbb{Z}
$$
so that its harmonic extension $f_n:S_n\to\mathbb{Z}$ satisfies (note the similarity of conditions 2-5),
\begin{enumerate}
\item The range of $f_n$ includes each integer in the set $[-n,n+1]$;
\item For $0<i<n$, we can estimate $f_n(UL_n^i)$ by
$$
f_n(UL_n^i)=(-1)^i\big(f_n(UL_n^0)-4i\cdot f_n(UL_{n-1}^0)\big)\pm\frac{i}{i+1}f_n(UL_{n-1}^0)\footnote{By $a=b\pm c$, we mean $b-c\leq a\leq b+c$}
$$
and
$$
f_n(UL_n^n)=(-1)^n\big(f_n(UL_n^0)-4n\cdot f_n(UL_{n-1}^0)+f_n(UR_n^0)\big)\pm\frac{n}{n+1}f_n(UL_{n-1}^0);
$$
\item For $0<i<n$,
$$
f_n(UR_n^i)=(-1)^i\big(f_n(UR_n^0)-4i\cdot f_n(UR_{n-1}^0)\big)\pm\frac{i}{i+1}f_n(UR_{n-1}^0)
$$
and
$$
f_n(UR_n^n)=(-1)^n\big(f_n(UR_n^0)-4n\cdot f_n(UR_{n-1}^0)+f_n(UL_n^0))\big)\pm\frac{n}{n+1}f_n(UR_{n-1}^0);
$$
\item For $1<i<n$,
$$
f_n(LL_n^i)=(-1)^{i+1}\big(f_n(UL_{n-1}^0)-4(i-1)\cdot f_n(UL_{n-2}^0)\big)\pm\frac{i}{i+1}f_n(UL_{n-2}^0)
$$
and for $n>1$,
$$
f_n(LL_n^n)=(-1)^n\big(f_n(UL_{n-1}^0)-4n\cdot f_n(UL_{n-2}^0)+f_n(UR_{n-1}^0))\big)\pm\frac{n}{n+1}f_n(UL_{n-2}^0);
$$
\item For $1<i<n$,
$$
f_n(LR_n^i)=(-1)^{i+1}\big(f_n(UR_{n-1}^0)-4(i-1)\cdot f_n(UR_{n-2}^0)\big)\pm\frac{i}{i+1}f_n(UR_{n-2}^0)
$$
and for $n>1$,
$$
f_n(LR_n^n)=(-1)^n\big(f_n(UR_{n-1}^0)-4n\cdot f_n(UR_{n-2}^0)+f_n(UL_{n-1}^0))\big)\pm\frac{n}{n+1}f_n(UL_{n-2}^0);
$$
\item $f_n$ is injective.
\end{enumerate}
We will construct $\widetilde{f}_{n+1}$ (extending $\widetilde{f}_n$) so that its harmonic extension $f_{n+1}$ also satisfies the above properties.  We proceed in four stages, analyzing $f_{n+1}$ on the upper-left, upper-right, lower-left and lower-right edges of its domain, $S_{n+1}$.

Let $N_1$ be the largest negative integer not attained by $f_n$ and $N_2$ be the smallest positive integer not attained by $f_n$.  Let $N_3$ be the sum of the absolute values of all integers in the range of $f_n$.  We define
\begin{eqnarray*}
\widetilde{f}_{n+1}(LL_{n+1}^0)&:=&N_1 \\
\widetilde{f}_{n+1}(LR_{n+1}^0)&:=&N_2 \\
\widetilde{f}_{n+1}(UL_{n+1}^0)&:=&10^{N_2-N_1+N_3} \\
\widetilde{f}_{n+1}(UR_{n+1}^0)&:=&10^{10^{N_2-N_1+N_3}}
\end{eqnarray*}
and $\widetilde{f}_{n+1}=\widetilde{f}_n$ for all points where $\widetilde{f}_n$ is defined.  Therefore $f_{n+1}=f_n$ on $S_n$.  By construction (of $\widetilde{f}_{n+1}$), property one is satisfied for $f_{n+1}$. \\

\noindent {\em Step 2.1 (Upper-left edge):}
We show that property $2$ is satisfied by induction (on $i$).  Note that $f_{n+1}$ satisfies property two at $UL_{n+1}^1$ because
$$
f_{n+1}(UL_{n+1}^1)=4f_{n+1}(UL_n^1)-f_{n+1}(UL_{n+1}^0)-f_{n+1}(UL_{n-1}^0)-f_{n+1}(LL_{n-1}^0).
$$
Now suppose $1\leq i<n$ and $f_{n+1}$ satisfies property two at $UL_{n+1}^j$ for each $0\leq j\leq i$.  Using the harmonic property,
\begin{equation}\label{estimate1}
f_{n+1}(UL_{n+1}^{i+1})=4f_{n+1}(UL_n^i)-f_{n+1}(UL_{n+1}^i)-f_{n+1}(UL_{n-1}^i)-f_{n+1}(UL_{n-1}^{i-1}).
\end{equation}
%
%

By property two and the induction hypothesis,
\begin{eqnarray*}
f_{n+1}(UL_n^i)&=&(-1)^i\big(f_{n+1}(UL_n^0)-4i\cdot f_{n+1}(UL_{n-1}^0)\big)\pm\frac{i}{i+1}f_{n+1}(UL_{n-1}^0) \\
f_{n+1}(UL_{n+1}^i)&=&(-1)^i\big(f_{n+1}(UL_{n+1}^0)-4i\cdot f_{n+1}(UL_n^0)\big)\pm\frac{i}{i+1}f_{n+1}(UL_n^0).
\end{eqnarray*}
By construction of $f_{n+1}(UL_{n-1}^0)$, $\frac{f_{n+1}(UL_{n-1}^0)}{f_{n+1}(UL_n^0)}<10^{-n}$.  Therefore (recalling that $i<n$),
\bae
&\left|-f_{n+1}(UL_{n-1}^i)-f_{n+1}(UL_{n-1}^{i-1})-(-1)^i4i\cdot f_{n+1}(UL_{n-1}^0)\pm\frac{i-1}{i}f_{n+1}(UL_{n-1}^0)\right|\\&<(4n+5)f_{n+1}(UL_{n-1}^0) \\
&<\frac{1}{n^2+1}f_{n+1}(UL_n^0).
\eae
Combining this with (\ref{estimate1}) shows that property two holds at $UL_{n+1}^{i+1}$.  By induction, property two holds for $0\leq i<n$.  A similar argument shows the case $i=n$.  \\


\noindent {\em Step 2.2 (Upper-right edge):} Property three follows by a similar argument to property two using $UR_{n+1}^i$ instead of $UL_{n+1}^i$. \\

\noindent {\em Step 2.3 (Lower-left edge):} Note that if $n>1$,
\begin{align}
f_{n+1}(LL_n^1)&=4f_{n+1}(LL_{n-1}^0)-f_{n+1}(LL_n^0)-f_{n+1}(LL_{n-2}^0)-f_{n+1}(UL_{n-1}^0) \\
&=-f_{n-1}(UL_{n-1}^0)\pm3n^2. \label{estimate2}
\end{align}
The rest of property four follows by a similar argument to property two. \\

\noindent {\em Step 2.4 (Lower-right edge):} Property five follows from a similar argument to property four. \\

\noindent {\em Step 2.5 (Injectivity):} Recall that $f_{n+1}$ is injective when restricted to $S_n$.  We need only see that it remains injective on its whole domain $S_{n+1}$.  Property two shows that $f_{n+1}$ is injective when restricted to $UL(n+1)$ (and properties 3-5 show that it is injective when restricted to $UR(n+1)$, $LL(n+1)$ and $LR(n+1)$, respectively).

If $(x,y)\in UL(n+1)$, then property two shows that
$$
\left|f_{n+1}(x,y)\right|>\left|f_{n+1}(UL_{n+1}^0)-4n\cdot f_{n+1}(UL_n^0)\right|>2\cdot\max_{(a,b)\in S_n}\left|f_{n+1}(a,b)\right|
$$
by construction of $f_{n+1}(UL_{n+1}^0)$ (which is $10^{\sum_{(a,b)\in S_n}\left|f_{n+1}(a,b)\right|}$).  Therefore $f_{n+1}$ is injective when restricted to $S_n\cup UL(n+1)$.  Similarly it is injective when restricted to $S_n\cup UR(n+1)$, $S_n\cup LL(n+1)$ and $S_n\cup LR(n+1)$.

We claim that $f_{n+1}$ is injective on all of $S_{n+1}$.  By properties 2-5, \\ \\
\centerline{\begin{tabular}{lll}
If $(x,y)\in UL(n+1)$ & then & $f_{n+1}(UL_{n+1}^0)=2^{\pm1}f_{n+1}(UL_{n+1}^0)$;\footnote{here $a=b^{\pm1}c$ means $\frac{c}{b}\leq a\leq bc$.} \\
If $(x,y)\in UR(n+1)$ & then & $f_{n+1}(UR_{n+1}^0)=2^{\pm1}f_{n+1}(UR_{n+1}^0)$; \\
If $(x,y)\in LL(n+1)$ & then & $f_{n+1}(LL_{n+1}^0)=2^{\pm1}f_{n+1}(UL_n^0)$; \\
If $(x,y)\in LR(n+1)$ & then & $f_{n+1}(LR_{n+1}^0)=2^{\pm1}f_{n+1}(UR_n^0)$.
\end{tabular}} \\ \\
Recall that, by construction,
$$
f_{n+1}(LR_{n+1}^0)>10^{f_{n+1}(LL_{n+1}^0)}>10^{10^{f_{n+1}(UR_n^0)}}>10^{10^{10^{f_{n+1}(UL_n^0)}}}.
$$
Thus, $f_{n+1}$ is injective on $S_{n+1}$ (i.e., property six holds). \\

\noindent {\em Step 3 (Construction of labeling):} By induction, we can construct the functions $\widetilde{f}_n$ and $f_n$ for all $n$ so that they have the six properties above.  Since $f_{n+1}$ extends $f_n$ we can define $f(x,y)$ to be $f_n(x,y)$ for any $n$ large enough that $(x,y)\in S_n$.  By construction, $f_{n+1}$ is harmonic on $S_n$ so, by properties one and six, $f$ is a harmonic labeling.
\end{proof}
\begin{thm}
There is a harmonic labeling for $\Z^d$ for all $d\in\mathbb{N}$.
\begin{proof}[Proof sketch]
After reading the proof of theorem \ref{thm:ztwolabeling} one can see how to generalize the proof for $d>2$ but at the same time it seems that writing a rigorous proof is rather complicated and tedious. The labeling of $\Z^d$ is spanned by its values on $\Z^{d-1}\times\{0\}\cup\Z^{d-1}\times\{-1\}$. We construct $\phi$ on the boundary of $[-n,n]^{d-1}$. For $d=2$ we added the rapidly increasing sequence in two points each step (the boundary of $[-n,n]$), for $d=3$ we introduce the sequence to the boundary of a rectangle and so on. Lets take $d=3$ as an example of how to continue. By introducing the rapid sequence on a rectangle boundary of one slice and a filling sequence on another slice we can extend $\phi$ to a pyramid shape by using the harmonic property. Now Each order of magnitude dominates a face of the pyramid, where again the equations are different for each vertex. Proving the difference of equations is the technical part to rigorize, but it is not hard for one to convince himself the construction works for all dimensions. The trick is to take the increasing sequence blow fast enough such that the increase will be fast enough to beat the polynomial number of elements of the lower magnitudes.
\end{proof}
\end{thm}

\section{Open problems}\label{sec:openprob}
\begin{enumerate}
\item\label{item:spanningquestion} Geometric properties of label spanning set. We say that a set $V'\subset V(G)$ is label spanning if using the harmonic property, given any values of $\phi$ on $V'$ we can extend $\phi$ to the entire graph. We saw several examples of finite label spanning sets (Cross, $\Z$ and the ladder). The first interesting question is are there any graphs other then several disconnected copies of $\Z$ which admit harmonic labeling spanned by a finite set? What can one tell about the graph from its spanning set?
\item Growth rate of harmonic labeling. Set some vertex as the origin $0$. Look at the set $A_n=\{\phi(x):d(x,0)\leq n\}$, where $d$ is the graph distance. What can one learn about a graph that admits harmonic labeling by the growth rate of $A_n$?
\item Are there any infinite graphs $G$ such that $G\times \Z$ doesn't admit harmonic labeling. Does any $G\times \Z$ where $G$ is infinite transitive admit harmonic labeling?
\item Is there a Cayley graph, which is not a finite extension of $\Z$, that doesn't admit harmonic labeling? What about the Heisenberg group? Other natural graphs to consider are hyperbolic lattices.
\item In Remark \ref{rmk:injecladder} we saw an example of a graph that does not admit harmonic labeling but has an injective harmonic function to $\Z$. Does any vertex transitive graph admit an injective harmonic function to $\Z$? Note that Dropping the injectivity constraint is easier as every graph that doesn't have a finite label spanning set admits a surjective harmonic function to $\Z$.
\end{enumerate}

\bibliography{Harmonic_Labeling_Edit_12}
\bibliographystyle{alpha}
\end{document}